\documentclass{amsart}
\usepackage{amsmath,amssymb,amsfonts,latexsym}
\usepackage[latin1]{inputenc}
\usepackage{amsthm}
\usepackage{verbatim}
\usepackage{color}

\theoremstyle{plain}
\newtheorem{theorem}{Theorem}[section]
\newtheorem{lemma}[theorem]{Lemma}
\newtheorem{proposition}[theorem]{Proposition}
\newtheorem{corollary}[theorem]{Corollary}

\theoremstyle{definition}

\newtheorem{remark}[theorem]{Remark}

\numberwithin{equation}{section}

\newcommand{\bo}{{\rm O}}
\newcommand{\ds}{\displaystyle}
\newcommand{\dint}{\ds\int}
\newcommand{\dsum}{\ds\sum}
\newcommand{\dx}[1]{\; {\rm d}#1}
\newcommand{\dm}[1]{{\, \rm d}#1}
\newcommand{\eqskip}{ \vspace*{2mm}\\ }
\newcommand{\R}{\mathbb{R}}
\newcommand{\N}{\mathbb{N}}

\begin{document}

\title[Summation formulae for the perturbed harmonic oscillator]{Summation formula inequalities for eigenvalues of the perturbed
harmonic oscillator}

\author{Pedro Freitas}
\address{Department of Mathematics, Faculty of Human Kinetics \&
Group of Mathematical Physics, Faculdade de Ci\^{e}ncias da Universidade de Lisboa, Campo Grande, Edif\'{\i}cio C6
1749-016 Lisboa, Portugal}
\email{psfreitas@fc.ul.pt}

\author{James B. Kennedy}
\address{Institute of Analysis, University of Ulm, Helmholtzstr.~18, D-89069 Ulm, Germany \&
Institute of Analysis, Dynamics and Modelling, University of Stuttgart, Pfaffenwaldring 57, D-70569 Stuttgart, Germany}

\email{james.kennedy@mathematik.uni-stuttgart.de}



\thanks{J.B.K. was supported by a fellowship of the Alexander
von Humboldt Foundation, Germany. Both authors were partially supported by FCT's project PEst-OE/MAT/UI0208/2011}

%\date{\today}

%\keywords{Eigenvalues, Trace formulae, Perturbed harmonic oscillator, Hermite polynomials}

\begin{abstract}
We derive explicit inequalities for sums of eigenvalues of one-dimensional Schr\"{o}dinger operators on the whole line.
In the case of the perturbed harmonic oscillator, these bounds converge to the corresponding trace formula in the limit as the
number of eigenvalues covers the whole spectrum.
\end{abstract}

%\subjclass[2000]{}

\maketitle

\section{Introduction}
\label{sec:intro}

Consider the eigenvalue equation
\begin{equation}
\label{general}
-u''(x) + V(x) u(x) = \lambda u(x), \;\; x\in (a,b) \subseteq \R,
\end{equation}
associated with a one-dimensional Schr\"odinger operator $H = -d^2/dx^2+V$, where the potential $V:(a,b) \to \R$, and
the boundary condition if $(a,b) \neq \R$, are chosen such that the spectrum consists of a discrete sequence of eigenvalues $\{\lambda_k\}$.
One possible way of linking the behaviour of this sequence to properties of the potential $V$ is via a regularized trace formula for
the sum of the eigenvalues. The classical example is the formula attributed to Gelfand and Levitan,
which, if we take $(a,b) = (0,\pi)$ with Dirichlet boundary conditions on the endpoints, reads
\begin{equation}
\label{classtr}
	\sum_{k=1}^\infty \left[\lambda_k - k^2 - \frac{1}{\pi}\int_0^\pi V(x)\dx{x} \right] 
	= \frac{1}{2\pi}\int_0^\pi V(x)\dx{x} - \frac{V(0)+V(\pi)}{4}
\end{equation}
(see, e.g., the book \cite{lesa}, also for other similar formulae). Since the values $k^2$ are in fact the eigenvalues of the 
Dirichlet Laplacian, that is, the corresponding Schr\"odinger operator with zero potential, this is a comparison between the 
eigenvalues of the operators $H$ and $H_0:=-d^2/dx^2$.

More recently it has also been shown that an analogous trace formula holds for the eigenvalues of~\eqref{general} on the 
whole line $(a,b) = \R$~\cite{abp,puso}. The comparison case is now provided by the quantum harmonic oscillator
\begin{equation}
\label{qho}
-u''(x) + x^2 u(x) = \lambda u(x), \;\; x\in\R,
\end{equation}
whose eigenvalues are given by $\lambda_k^0 = 2k+1$ for $k \in \N$. Writing the potential in \eqref{general} as $V(x) = x^2 + q(x)$, that is,
as a perturbed harmonic oscillator, if the perturbation $q:\R \to \R$ is small enough in an appropriate sense,
then the eigenvalues of~\eqref{general}, which we denote by $\lambda_k$ for $k \in \N$, satisfy the trace formula
\begin{equation}
\label{photrace}
\dsum_{k=0}^{\infty}\left[ \lambda_{k}-\lambda_{k}^{0}-\frac{\ds 1}
{\ds \pi \sqrt{\lambda_{k}^{0}}}\dint_{\R} q(x) \dx{x}\right] = -\frac{\ds Z_0(1/2)}{\ds \pi}
\dint_{\R} q(x) \dx{x},
\end{equation}
where
\begin{equation}
\label{basespeczeta}
 Z_0(s) =(1-2^{-s})\zeta(s) = \sum_{k=1}^{\infty} \frac{\ds 1}{\ds \left(\lambda_{k}^{0}\right)^{s}}
\end{equation}
is the spectral zeta function associated with \eqref{qho}, the second equality being valid for $\textrm{Re}\,{s}>1$,
and $\zeta(\,.\,)$ is the Riemann zeta function; see~\cite[Theorem~2]{abp} or~\cite[Eq.~(1.12)]{puso}.
We refer to \cite{sapo} for a wide-ranging general survey on the theory of regularized traces.

In a separate paper \cite{frke} we show that formula~\eqref{classtr} is in fact the limit as $n \to \infty$ of a sequence of inequalities
for the (finite) sums of the first $n$ eigenvalues given in terms of the Fourier coefficients of the potential, and that~\eqref{classtr}
can be proved by combining these inequalities with knowledge of the asymptotic behaviour of the eigenvalues and eigenfunctions~\cite{frke}.
In the present paper, which may be viewed as a continuation of~\cite{frke}, we show that a similar family of inequalities is valid
for the perturbed harmonic oscillator assuming that the perturbation $q$ is non-negative and of finite $L^1(\R)$-norm. More precisely, we
shall prove in Theorem~\ref{mainth1} below that there is a sequence of inequalities of the form
\begin{displaymath}
\dsum_{k=0}^{n}\left[ \lambda_{k}-\lambda_{k}^{0}-\frac{\ds 1}
{\ds \pi \sqrt{\lambda_{k}^{0}}}\dint_{\R} q(x) \dx{x}\right]\leq \frac{\ds \chi_{n}}{\ds \pi}
\dint_{\R} q(x) \dx{x}
\end{displaymath}
for all $n \in \N$ if $V(x)=x^2+q(x)$ with $0 \leq q \in L^1(\R)$, where the sequence $\chi_n$, which is given explicitly, depends 
only on properties of the eigenfunctions and 
eigenvalues of the quantum harmonic oscillator~\eqref{qho} and converges to $-Z_0(1/2)$ like $\bo(1/\sqrt{n})$ as $n\to\infty$. 
A similar sequence of bounds will also be shown to hold for a certain class of negative or indefinite potentials (see Theorem~\ref{mainth2}),
and although the corresponding bounding sequence we obtain is larger than $\chi_n$, it is still explicit, and the order of convergence to 
the known trace formula remains $\bo(1/\sqrt{n})$.

These results will be established via test function methods, using for this purpose the eigenfunctions of \eqref{qho} in a 
suitable Rayleigh quotient expression for the eigenvalues of the perturbed harmonic oscillator, and then combining this with properties of 
Hermite polynomials to analyze the resulting expression. We believe one of these properties, namely Lemma~\ref{lemma:hn}, which provides
an upper bound for the function $e^{-x^2}\left[H_{n+1}^{2}(x)-H_{n}(x)H_{n+2}(x)\right]$ to be new and interesting in its own right.

In fact, these results---and the corresponding proofs---differ from those in~\cite{frke} in that for them we do not use a decomposition of
the potential in terms of the eigenfunctions of the unperturbed problem. However, such an approach is also possible in this case and we 
carry it out to obtain a different type of bound; see Theorem~\ref{genpot}. For this particular result we assume that 
$V \in L^2(\R,e^{-x^2}\dm{x})$, that is, that the potential is no longer necessarily a perturbation of $x^2$, but rather more 
generally merely square integrable with respect to the weighted $L^2$-measure most naturally associated with the problem~\eqref{qho}. 
The resulting bounds (which are once again explicit) are expressed in terms of the Fourier-like coefficients of $V$ expanded as a sum of
Hermite polynomials. These are actually stronger than Theorems~\ref{mainth1} and~\ref{mainth2}, as the only inequality used now is
that which arises from the substitution of test functions in the Rayleigh quotient (see Remark~\ref{rem:genpot}(i)). However, now the 
finite sums converging to the the left-hand side of the trace formula~\eqref{photrace} do not appear in a natural way; this will then be 
derived as a simple corollary by writing the potential $V(x)$ as $x^2+q(x)$ and using the Fourier coefficients for $q$ instead.

We also generalize Theorems~\ref{mainth1} and~\ref{mainth2} to obtain bounds on sums of powers of the eigenvalues in Section~\ref{sec:power}.

\section{Schr\"odinger operators on the real line}
\label{sec:prelim}

Throughout this paper we will consider one-dimensional Schr\"odinger operators on the real line, that is, associated with the
equation~\eqref{general} for $x \in \R$, where the potential $V:\R \to \R$ is a locally measurable function on which we will impose
various (and varying) assumptions. We will always assume that $V(x) \to \infty$ as $|x| \to \infty$, so that the operator associated with the problem \eqref{general} considered as an operator on $L^2(\R)$ has discrete spectrum, and we will in general denote the associated eigenvalues by $\lambda_0 < \lambda_1 \leq \ldots \to \infty$.

As is well known, the eigenvalues of the quantum harmonic oscillator \eqref{qho}, which will play the role of our ``default'' problem, 
are given by $\lambda_{k}^{0}=2k+1$ for $k\in\N$, with corresponding eigenfunctions $\psi_{k}(x) = e^{-x^2/2}H_{k}(x)$, which form an orthonormal basis of $L^2(\R)$. Here $H_{k}$ denotes the $k^{\rm th}$ Hermite polynomial (see, e.g., \cite[Ch.~5]{szeg}).

Of particular interest to us will be the perturbed harmonic oscillator
\begin{equation}
\label{pertharm}
-u''(x) + \left[ x^2 + q(x)\right] u(x) = \lambda u(x), \;\; x\in\R,
\end{equation}
which is easily seen to have discrete spectrum if $q \in L^p(\R)$ for some $p \in [1,\infty]$.

For a general potential $V:\R\to\R$, we can characterize the associated eigenvalues via classical variational methods. Denoting by 
$\varphi \in H^1(\R) \cap L^2(\R,V(x)\dm{x})$ an arbitrary test function, we let
\begin{equation}
\label{rayleigh}
	\mathcal{R}[V,\varphi]:= \frac{\dint_{\R} \left(\varphi'(x)\right)^2\dx{x}+\dint_{\R}V(x)\varphi^2(x)\dx{x}}
{\dint_{\R} \varphi^2(x)\dx{x}}
\end{equation}
be the Rayleigh quotient associated with (the Schr\"odinger operator with potential) $V$ at $\varphi$. A standard generalization of
the usual minimax formula for eigenvalues states that if $\varphi_0,\ldots,\varphi_n$ is a collection of $n+1$ such functions orthogonal in $L^2(\R)$, for any $n\in\N$, then
\begin{displaymath}
	\sum_{k=0}^{n} \lambda_k \leq \sum_{k=0}^{n} \mathcal{R}[V,\varphi_k]
\end{displaymath}
(see, e.g., \cite{band}), with equality being achieved when the $\varphi_k$ are the first $n+1$ eigenfunctions. For us the most natural choice of test functions will be the eigenfunctions $\psi_k$ of the quantum harmonic operator.

\section{Bounds for the perturbed harmonic oscillator with a non-negative perturbation}
\label{sec:pho}

In this section we will state and prove our main theorem, obtaining the aforementioned finite version of the trace formula~\eqref{photrace} for the general perturbed harmonic oscillator \eqref{pertharm}.

\begin{theorem}
\label{mainth1}
Let $q$ be a non-negative potential defined on the real line having finite $L^{1}(\R)$ norm. Then the eigenvalues of \eqref{pertharm} satisfy the inequalities
\begin{equation}
\label{maineq}
\dsum_{k=0}^{n}\left[ \lambda_{k}-\lambda_{k}^{0}-\frac{\ds 1}
{\ds \pi \sqrt{\lambda_{k}^{0}}}\dint_{\R} q(x) \dx{x}\right]\leq \frac{\ds \chi_{n}}{\ds \pi}
\dint_{\R} q(x) \dx{x},\;\; n=0,1,\cdots,
\end{equation}
where
\begin{equation}
\label{chi}
\chi_{n} = \left\{
\begin{array}{ll}
\frac{\ds 2n+3}{\ds n+1}\frac{\ds \Gamma\left(\frac{\ds n}{\ds 2}+1\right)}
{\ds \Gamma\left(\frac{\ds n+1}{\ds 2}\right)}-\dsum_{k=0}^{n}\frac{\ds 1}
{\ds \sqrt{\lambda_{k}^{0}}},
& n \mbox{ odd,}\eqskip
(n+1)\frac{\ds \Gamma\left(\frac{\ds n+1}{\ds 2}\right)}{\ds \Gamma\left(\frac{\ds n}{\ds 2}+1\right)}
-\dsum_{k=0}^{n}\frac{\ds 1}{\ds \sqrt{\lambda_{k}^{0}}},
& n \mbox{ even.}
\end{array}
\right.
\end{equation}
Furthermore, $\chi_{n} = -Z_{0}(1/2) + \bo(1/\sqrt{n})$, where $Z_{0}(s) = (1-2^{-s})\zeta(s)$.
\end{theorem}

\begin{remark}
(i) It is essential for our method of proof that $q$ be non-negative. In Theorem~\ref{mainth2} below, we weaken this assumption
and obtain a slightly weaker set of inequalities which nevertheless still converge in the limit to the trace formula~\eqref{photrace}
with the same order of convergence $\bo(1/\sqrt{n})$. It is not clear if the inequalities \eqref{maineq} are true for arbitrary
$q \in L^1(\R)$; the trace formula \eqref{photrace} is itself currently only known to hold under stronger assumptions on $q$: in
\cite{abp} a certain rate of decay of $q$ at infinity is assumed, and in \cite{puso} it is assumed $q$ has compact support. We
remark however that having convergence of order $\bo(1/\sqrt{n})$ is most probably optimal, since this is the rate at which we
have convergence of the sequence whose limit defines $\zeta(1/2)$ (cf.~\eqref{zetaexp} and \eqref{z0}).

(ii) There do not exist corresponding lower bounds for \emph{finite} sums of eigenvalues: for any fixed $n \geq 0$ is it always
possible to find a function $0 \leq q \in L^1(\R)$ for which the left-hand side of \eqref{maineq} is arbitrarily large negative;
see Proposition~\ref{prop:lower} below. However, for a \emph{fixed} potential it is a natural question as to whether we can recover
a lower bound valid in the asymptotic limit. Indeed, it might be possible to extend our result to give a new proof of the trace
formula \eqref{photrace} for a different class of (non-negative) potentials $q$ from those considered
in~\cite{abp,puso}, namely $q \in L^1(\R)$. The idea would be to argue as in \cite{dikii} (or \cite{frke}), to show
that the degree of ``error'' which arises from using the eigenfunctions $\psi_k$ of the unperturbed problem as test
functions becomes asymptotically small as $k \to \infty$: denoting by $\varphi_k$ the eigenfunction associated with
$\lambda_k$ (corresponding to the potential $V(x)=x^2+q(x)$), we see that the trace formula holds whenever
\begin{equation}
\label{rayleigh-lim}
	\lim_{n \to \infty} \dsum_{k=0}^{n} \left(\mathcal{R}[x^2+q(x),\varphi_k] - \mathcal{R}[x^2+q(x),\psi_k]\right) = 0,
\end{equation}
since by definition $\mathcal{R}[x^2+q(x),\varphi_k]=\lambda_k$. We can rewrite \eqref{rayleigh-lim} as a type of ``change of basis'' formula
\begin{displaymath}
	\lim_{n \to \infty} \dsum_{k=0}^{n} \big(\langle \varphi_k, H\varphi_k \rangle - \langle \psi_k, H\psi_k\rangle\big)=0,
\end{displaymath}
where $H:  D(H)\subset L^2(\R) \to L^2(\R)$ is the operator associated with the potential $x^2+q(x)$.
We \emph{expect} this to hold whenever the asymptotics for $\lambda_k$ and $\varphi_k$ are similar enough to those of
$\lambda_k^0$ and $\psi_k$, respectively, when $k \to \infty$. This is, however, likely to be a difficult problem,
and we shall not attempt an investigation of it here.
\end{remark}

For notational convenience, for $n\geq 0$ we define
\begin{equation}
\label{omega}
\omega_n:=\chi_n+\dsum_{k=0}^n \frac{1}{\sqrt{\lambda_k^0}}= \left\{
\begin{array}{ll}
\frac{\ds 2n+3}{\ds n+1}\frac{\ds \Gamma\left(\frac{\ds n}{\ds 2}+1\right)}
{\ds \Gamma\left(\frac{\ds n+1}{\ds 2}\right)},
& n \mbox{ odd,}\eqskip
(n+1)\frac{\ds \Gamma\left(\frac{\ds n+1}{\ds 2}\right)}{\ds \Gamma\left(\frac{\ds n}{\ds 2}+1\right)},
& n \mbox{ even,}
\end{array}
\right.
\end{equation}
and we also set $\omega_{-1}:= 0$.

\begin{proof}[Proof of Theorem~\ref{mainth1}]
Using the first $n+1$ eigenfunctions of the unperturbed harmonic oscillator \eqref{qho}, given by $\psi_k(x) =e^{-x^2/2}  H_k(x)$, 
$k=0,\ldots,n$, as test functions in the Rayleigh quotient \eqref{rayleigh} for $V(x)=x^ 2+q(x)$ yields
\begin{equation}
\label{basicsum}
\begin{array}{lll}
\dsum_{k=0}^{n} \lambda_{k} & \leq & \dsum_{k=0}^{n} \frac{\ds \dint_{\R}\left[\frac{\ds d}{\ds dx}
\left[e^{-x^2/2}H_{k}(x)\right]\right]^2
 +
\left[x^2+q(x)\right]e^{-x^2}H_{k}^{2}(x)\dx{x}}{\ds \dint_{\R}e^{-x^2}H_{k}^{2}(x)\dx{x}}\eqskip
& = & \dsum_{k=0}^{n} \lambda_{k}^{0} + \dint_{\R} e^{-x^2} q(x)\dsum_{k=0}^{n}\frac{\ds 1}
{\ds 2^k k!\sqrt{\pi}}H_{k}^{2}(x)\dx{x}
\end{array}
\end{equation}
From basic properties of Hermite polynomials we have the identity
\begin{equation}
\label{sumH}
\dsum_{k=0}^{n}\frac{\ds 1}{\ds 2^k k!}H_{k}^{2}(x)=\frac{\ds 1}{\ds 2^{n+1}n!}\left[ H_{n+1}^{2}(x)
-H_{n}(x)H_{n+2}(x)\right].
\end{equation}
This arises in the context of Tur\'an's inequality for Hermite polynomials (cf. \cite[p. 404]{szeg1}), and can easily be derived directly
by induction in $n$---see also, for instance,~\cite[p.~106]{szeg}. By using the estimate of the function
\begin{equation}
\label{h}
h_{n}(x) := e^{-x^2}\left[H_{n+1}^{2}(x)-H_{n}(x)H_{n+2}(x)\right]
\end{equation}
given in Lemma~\ref{lemma:hn} below in \eqref{sumH} and inserting this into \eqref{basicsum}, we obtain
\begin{multline}
\label{comb}
\dsum_{k=0}^{n} \lambda_{k} \leq \dsum_{k=0}^{n} \lambda_{k}^{0} + \dint_{\R} e^{-x^2} q(x)\dsum_{k=0}^{n}\frac{\ds 1}
{\ds 2^k k!\sqrt{\pi}}H_{k}^{2}(x)\dx{x} \\
\leq \dsum_{k=0}^{n} \lambda_{k}^{0} + \frac{\omega_n}{\pi}\dint_{\R} q(x)\dx{x},
\end{multline}
which upon rearranging yields \eqref{maineq}.

We now give the (routine) proof that $\chi_n = -Z_0(1/2) + \bo(1/\sqrt{n})$ as $n \to \infty$. We first note that
\begin{equation}
\label{zetaexp}
\zeta(s) = \sum_{k=1}^n k^{-s}+ s\int_n^\infty \frac{\lfloor x\rfloor - x + \frac{1}{2}}{x^{s+1}} \dx{x} + 
	\frac{n^{1-s}}{s-1} - \frac{1}{2n^{s}},
\end{equation}
valid for $s>0$ (see \cite{titc}, Eq.~(3.5.3), pp.~49-50). Setting $s=1/2$ and passing to the limit as $n \to \infty$,
this means we can write 
\begin{equation}
\label{z0}
-Z_0(1/2) = -\left(1-\frac{1}{\sqrt{2}}\right) \zeta(1/2)
= \left(1-\frac{1}{\sqrt{2}}\right)\lim_{n \to \infty} a_n,
\end{equation}
where for ease of notation we have set
\begin{equation}
\label{an}
a_n:= 2\sqrt{n} - \sum_{k=1}^n\frac{1}{\sqrt{k}}
\end{equation}
for $n \geq 1$. Now, recalling that $\lambda_k^0 = 2k+1$ for $k \in \N$, we have
\begin{equation}
\label{eulersum}
\chi_n = \omega_n - \sum_{k=0}^n \frac{1}{\sqrt{2k+1}} = \omega_n - \left(1-\frac{1}{\sqrt{2}}\right)\sum_{k=1}^{n} \frac{1}{\sqrt{k}} - \sum_{k=n+1}^{2n+1} \frac{1}{\sqrt{k}};
\end{equation}
we wish to show that this converges to $-Z_0(1/2)$ as $n\to\infty$. We first establish that
\begin{equation}
\label{omegaconv}
\omega_n = \sqrt{2n} + \bo\left(\frac{1}{\sqrt{n}}\right),
\end{equation}
using the following asymptotics for the quotient of two gamma functions (see~\cite{abrasteg}, formula $6.1.47$, for instance):
\begin{equation}
\label{gammaexp}
 \frac{\Gamma\left(z+\frac{1}{2}\right)}{\Gamma(z)} = \sqrt{z} + \bo\left(\frac{\ds 1}{\ds \sqrt{z}}\right),
\end{equation}
for large $z$. For $n$ odd we obtain
\begin{displaymath}
\omega_n = \frac{2n+3}{n+1} \left[\sqrt{\frac{n+1}{2}} + \bo\left(\frac{1}{\sqrt{n}}\right)\right]
= \sqrt{2n} + \bo\left(\frac{1}{\sqrt{n}}\right).
\end{displaymath}
A similar calculation when $n$ is even gives
\begin{displaymath}
\omega_n = 2\, \frac{\Gamma\left(\frac{n+3}{2}\right)}{\Gamma\left(\frac{n}{2}+1\right)}
= 2\left[\sqrt{\frac{n}{2}+1} + \bo\left(\frac{1}{\sqrt{n}}\right)\right],
\end{displaymath}
proving \eqref{omegaconv}. Next, we observe that
\begin{displaymath}
\sum_{k=n+1}^{2n+1} \frac{1}{\sqrt{k}} = 2(\sqrt{2}-1)\sqrt{n} + \bo\left(\frac{1}{\sqrt{n}}\right)
\end{displaymath}
for large $n$, as can be seen, for example, by noting that
\begin{displaymath}
\dint_{n+1}^{2n+2} \frac{1}{\sqrt{x}}\dx{x} \leq \dsum_{k=n+1}^{2n+1}\frac{1}{\sqrt{k}} \leq
\dint_{n+1}^{2n+2} \frac{1}{\sqrt{x-1}}\dx{x}
\end{displaymath}
and evaluating the integrals. Substituting these two estimates into~\eqref{eulersum} yields
\begin{displaymath}
\begin{aligned}
\chi_n &= \sqrt{2n} - \left(1-\frac{1}{\sqrt{2}}\right)\dsum_{k=1}^{n} \frac{1}{\sqrt{k}} - 2\left(\sqrt{2}-1\right)\sqrt{n} +
\bo\left(\frac{1}{\sqrt{n}}\right)\\
&= \left(1-\frac{1}{\sqrt{2}}\right)a_n + \bo\left(\frac{1}{\sqrt{n}}\right).
\end{aligned}
\end{displaymath}
Letting $s$ equal $1/2$ in~\eqref{zetaexp} and using $-1< \lfloor x\rfloor - x\leq 0$ we obtain
\begin{displaymath}
 -\frac{1}{\sqrt{n}}< \zeta(1/2) + a_{n} \leq 0,
\end{displaymath}
from which it follows that
\begin{displaymath}
 \chi_n = -Z_{0}(1/2)+\bo\left(\frac{1}{\sqrt{n}}\right),
\end{displaymath}
as desired.
\end{proof}

\begin{lemma}
\label{lemma:hn}
The function $h_{n}$ defined by \eqref{h} is positive and satisfies
\begin{displaymath}
h_{n}(x)\leq \left\{\begin{array}{ll}
\frac{\ds 4^{n+1}}{\ds 2\pi}\frac{\ds 2n+3}{\ds n+1}\Gamma^{2}\left(\frac{\ds n}{\ds 2}+1\right), & n \mbox{ odd,}\eqskip\eqskip
\frac{\ds 4^{n+1}}{\ds 2\pi}(n+1) \Gamma^{2}\left(\frac{\ds n+1}{\ds 2}\right), & n \mbox{ even.}
            \end{array}
\right.
\end{displaymath}
\end{lemma}

\begin{proof}
Positivity of $h_{n}$ is a direct consequence of \eqref{sumH}. Taking derivatives in $x$ and using the property $H_{n}'(x) = 2n H_{n-1}(x)$ yields
\begin{displaymath}
\begin{array}{lll}
h_{n}'(x) & = & e^{-x^2}\left\{ -2x\left[H_{n+1}^{2}(x)-H_{n}(x)H_{n+2}(x)\right]\right.\eqskip
& & \hspace*{0.5cm} \left.+ 2H_{n+1}(x)H_{n+1}'(x)-H_{n}'(x)H_{n+2}(x)-H_{n}(x)H_{n+2}'(x)\right\}\eqskip
& = & e^{-x^2}\left\{ 2H_{n+1}(x)\left[-x H_{n+1}(x)+2(n+1)H_{n}(x)\right] \right. \eqskip
& & \hspace*{0.5cm} +2x H_{n}(x)H_{n+2}(x)-2nH_{n-1}(x) H_{n+2}(x) \eqskip
& & \hspace*{1cm} \left.-2(n+2)H_{n}(x)H_{n+1}(x)\right\}\eqskip
& = & e^{-x^2}\left\{ 2H_{n+1}(x)\left[-x H_{n+1}(x)+n H_{n}(x)\right]\right.\eqskip
& & \hspace*{0.5cm} \left. +2 H_{n+2}\left[x H_{n}(x)-nH_{n-1}(x)\right]\right\}.
\end{array}
\end{displaymath}
Using the identity $H_{n+1}(x) = 2x H_{n}(x)-2nH_{n-1}(x)$ in the above expression yields
\begin{displaymath}
h_{n}'(x) = -2e^{-x^2}H_{n}(x)H_{n+1}(x),
\end{displaymath}
which integrated between zero and $x$ becomes
\begin{displaymath}
\begin{array}{lll}
h_{n}(x)-h_{n}(0) & = & -2 \dint_{0}^{x} e^{-t^2}H_{n}(t)H_{n+1}(t)\dx{t}\eqskip
& = & -\frac{\ds 1}{\ds 2(n+1)}\dint_{0}^{x} e^{-t^2}\frac{\ds d}{\ds dt} H_{n+1}^{2}(t)\dx{t}\eqskip
& = & -\frac{\ds 1}{\ds 2(n+1)} \left[ e^{-x^2}H_{n+1}^{2}(x)-H_{n+1}^{2}(0)+ 2\dint_{0}^{x}
t e^{-t^2}H_{n+1}^{2}(t)\dx{t}\right].
\end{array}
\end{displaymath}
Noting that the terms which depend on $x$ on the right-hand side above are non-positive, we obtain
\begin{equation}\label{hnbound}
h_{n}(x)-h_{n}(0) \leq \frac{\ds 1}{\ds 2(n+1)}H_{n+1}^{2}(0).
\end{equation}
For odd $n$, $h_{n}(0) = H_{n+1}^{2}(0)$ and the above becomes
\begin{displaymath}
\begin{array}{lll}
h_{n}(x) & \leq & \frac{\ds 2n+3}{\ds 2n+2} H_{n+1}^{2}(0)\eqskip
& = & \frac{\ds 2n+3}{\ds 2n+2} \Gamma^{2}(n+2)\Gamma^{-2}\left(\frac{\ds n+3}{\ds 2}\right)\eqskip
& = & \frac{\ds 4^{n+1}}{\ds 2\pi}\frac{\ds 2n+3}{\ds n+1}\Gamma^{2}\left(\frac{\ds n}{\ds 2}+1\right).
\end{array}
\end{displaymath}
For even values of $n$ the right-hand side of \eqref{hnbound} vanishes and we obtain
\begin{displaymath}
h_{n}(x) \leq h_{n}(0) = \frac{\ds 4^{n+1}}{\ds 2\pi}(n+1)\Gamma^{2}\left(\frac{\ds n+1}{\ds 2}\right).
\end{displaymath}
\end{proof}

We will now construct an example showing that no lower bound of the same form as in Theorem~\ref{mainth1} is possible. 

\begin{proposition}
\label{prop:lower}
For any $n \geq 0$ and any $N>0$, there exists $0 \leq q \in L^1(\R)$ such that
\begin{displaymath}
	\dsum_{k=0}^{n}\left[ \lambda_{k}-\lambda_{k}^{0}-\frac{\ds 1} {\ds \pi \sqrt{\lambda_{k}^{0}}}\dint_{\R} 
	q(x) \dx{x}\right]\leq -N.
\end{displaymath}
\end{proposition}

Before giving the proof, we note two points: firstly, that there \emph{exists} a potential for which the corresponding first $n$ eigenvalues are arbitrarily large negative is trivial; the key point here is that $q$ satisfies the same assumptions as in Theorem~\ref{mainth1}. Secondly, the sum here has to be regularized, since for any $q \geq 0$ we automatically have $\lambda_k \geq \lambda_k^0$ for all $k \geq 0$.

\begin{proof}
Fix $n \geq 0$ and $N>0$. If we use the $n+1$ functions $\psi_k(x)=e^{-x^2/2}H_k(x)$ for $k=1,3,\ldots,2n+1$, as test functions in the Rayleigh quotient, then for any $0 \leq q \in L^1(\R)$ we obtain after a certain amount of rearranging
\begin{multline}
\label{neg-test}
	\dsum_{k=0}^{n}\left[ \lambda_{k}-\lambda_{k}^{0}-\frac{\ds 1} {\ds \pi \sqrt{\lambda_{k}^{0}}}\dint_{\R} 
	q(x) \dx{x}\right] \\ \leq C_n + \dsum_{k=0}^{n}\dint_{\R} e^{-x^2}q(x) \frac{H_{2k+1}^2(x)}{2^{k+1}(2k+1)!
	\sqrt{\pi}}\dx{x}- \dsum_{k=0}^{n}\frac{\ds 1} {\ds \pi \sqrt{\lambda_{k}^{0}}}\dint_{\R} q(x) \dx{x},
\end{multline}
where the constant
\begin{displaymath}
	C_n := \dsum_{k=0}^{n} \lambda_{2k+1}^0 - \dsum_{k=0}^{n} \lambda_k^0 \geq 0
\end{displaymath}
depends only on $n \geq 0$. We will show that we can find $q$ for which the first sum on the right-hand side of \eqref{neg-test} is arbitrarily small, while the second sum is arbitrarily large. The idea is to choose $q$ to have support in a very small neighbourhood of $0$ and use that all odd Hermite polynomials $H_{2k+1}$ satisfy $H_{2k+1}(0)=0$ (and hence are very small close to $0$). We start by fixing $K=K(n,N)>0$ large enough that
\begin{equation}
\label{k-choice}
	C_n + 1 - K\dsum_{k=0}^{n}\frac{\ds 1} {\ds \pi \sqrt{\lambda_{k}^{0}}} < -N
\end{equation}
and for given $\delta>0$, to be specified later, we choose $q_\delta(x):=K\delta^{-1}\chi_\delta(x)$, where $\chi_\delta$ is the indicator function of the set $[-\delta/2,\delta/2]$. Then obviously $q_\delta \geq 0$ has $L^1$-norm equal to $K$ for any $\delta>0$. Since, as mentioned, $H_{2k+1}^2(0)=0$ for all $k=0,\ldots,n$, and $H_{2k+1}^2$ is obviously continuous, for any $\varepsilon>0$, there exists $\delta=\delta(\varepsilon, n) >0$ such that
\begin{displaymath}
	0 \leq \frac{e^{-x^2} H_{2k+1}^2(x)}{2^{k+1}(2k+1)!\sqrt{\pi}} < \varepsilon
\end{displaymath}
for all $x \in [-\delta/2,\delta/2]$ and all $k=0,\ldots,n$. It follows that for this $\delta$, we have
\begin{displaymath}
	\dsum_{k=0}^{n}\dint_{\R} e^{-x^2}q_\delta (x) \frac{H_{2k+1}^2(x)}{2^{k+1}(2k+1)!
	\sqrt{\pi}}\dx{x} < \varepsilon(n+1) < 1,
\end{displaymath}
if we choose $\varepsilon < 1/(n+1)$. Inserting this estimate together with \eqref{k-choice} into \eqref{neg-test} yields the proposition.
\end{proof}

\section{Bounds for the perturbed harmonic oscillator with an integrable perturbation}
\label{sec:neg}

Here we generalize Theorem~\ref{mainth1} to allow for a class of perturbations $q$ which may now take on negative values. 
Although the resulting estimate is not quite as tight as in Theorem~\ref{mainth1}, we still have convergence to the trace formula~\eqref{photrace} at the same rate as before.

\begin{theorem}
\label{mainth2}
Given the function $q \in L^1(\R)$, suppose that there exists a non-negative constant $q_m$ for which $q(x) + q_m e^{-x^2}$ is non-negative for almost all real values of $x$. Then the eigenvalues of the corresponding perturbed harmonic oscillator~\eqref{pertharm} satisfy the inequalities
\begin{equation}
\label{negeq}
\dsum_{k=0}^{n}\left[ \lambda_{k}-\lambda_{k}^{0}-\frac{\ds 1}
{\ds \pi \sqrt{\lambda_{k}^{0}}}\dint_{\R} q(x) \dx{x}\right]\leq \frac{\ds \chi_{n}}{\ds \pi}
\dint_{\R} q(x) \dx{x} + \varepsilon_{n} \frac{q_m}{\sqrt{\pi}}
\end{equation}
for $n = 0,1,\ldots$, where
\begin{equation}
\label{negasymp}
\varepsilon_{n} = \omega_n - \sqrt{2}\,\frac{\Gamma\left(n+\frac{3}{2}\right)}{\Gamma(n+1)} \geq 0
\end{equation}
and $\chi_n$ and $\omega_n$ are given by~\eqref{chi} and~\eqref{omega}, respectively. Moreover, $\varepsilon_{n}= \bo(1/\sqrt{n})$ as $n \to \infty$.
\end{theorem}

\begin{proof}
We suppose $q_m \geq 0$ is as in the statement of the theorem, and mimic the proof of Theorem~\ref{mainth1} to obtain
\begin{displaymath}
\begin{split}
\dsum_{k=0}^{n} \lambda_k &\leq \dsum_{k=0}^{n} \lambda_k^0 + \dint_\R e^{-x^2}q(x)\dsum_{k=0}^{n}
	\frac{1}{2^k k!\sqrt{\pi}} H_k^2(x)\dx{x}\\
	&= \dsum_{k=0}^{n} \lambda_k^0 + \dint_\R e^{-x^2}\left[q(x)+q_m e^{-x^2}\right]
	\dsum_{k=0}^{n}\frac{1}{2^k k!\sqrt{\pi}} H_k^2(x)\dx{x}\\
	&\qquad\qquad -\frac{q_m}{\sqrt{\pi}}\dsum_{k=0}^{n} \frac{1}{2^k k!} \dint_\R e^{-2x^2}H_k^2(x)\dx{x}.
\end{split}
\end{displaymath}
Since $q(x)+q_m e^{-x^2} \in L^1(\R)$ is positive by assumption, we may proceed as in the proof of Theorem~\ref{mainth1} to obtain
\begin{displaymath}
\begin{split}
\dint_\R e^{-x^2}\left[q(x)+q_m e^{-x^2}\right] \dsum_{k=0}^{n}\frac{1}{2^k k!\sqrt{\pi}} H_k^2(x)\dx{x}
	&\leq \frac{\omega_n}{\pi}\dint_\R q(x)+q_m e^{-x^2}\dx{x}\\
	&= \frac{\omega_n}{\pi} \dint_\R q(x)\dx{x} + \frac{q_m}{\sqrt{\pi}}\,\omega_n.
\end{split}
\end{displaymath}
Meanwhile, since
\begin{displaymath}
\dint_\R e^{-2x^2} H_k^2(x)\dx{x} = 2^{k-\frac{1}{2}} \Gamma\left(k+\frac{1}{2}\right),
\end{displaymath}
we have
\begin{displaymath}
\dsum_{k=0}^{n}\frac{1}{2^k k!}\dint_\R e^{-2x^2}H_k^2(x)\dx{x}
	=\frac{1}{\sqrt{2}}\dsum_{k=0}^{n} \frac{\Gamma\left(k+\frac{1}{2}\right)}{k!}
	=\sqrt{2} \,\frac{\Gamma\left(n+\frac{3}{2}\right)}{\Gamma(n+1)}.
\end{displaymath}
Combining the above expressions yields \eqref{negeq}. The asymptotic behaviour of $\varepsilon_{n}$ is an immediate consequence
of \eqref{omegaconv} together with the expansion~\eqref{gammaexp}.

Although $\varepsilon_n$ can be computed explicitly, to see that it is positive we use the following easier, indirect argument: 
if for a given $q \in L^1(\R)$, \eqref{negeq} holds for some $q_m \geq 0$, then the above proof shows that it also holds for 
all $c \geq q_m$. This is only possible if $\varepsilon_n \geq 0$ for all $n\geq 0$.
\end{proof}

\section{A bound for a general potential in terms of Hermite polynomials}
\label{sec:genpot}

Here we will consider the general problem \eqref{general}, supposing only that the potential $V: \R \to \R$ admits a series expansion in 
terms of Hermite polynomials in the manner of an eigenfunction decomposition
\begin{displaymath}
V(x) = \dsum_{j=0}^{\infty} v_j H_j(x),
\end{displaymath}
where we now assume that $V(x) \in L^2(\R,e^{-x^2}\dm{x})$, or equivalently, since the $H_j$ form an orthonomal basis of $L^2(\R)$ with 
respect to this measure, that the sequence $v_j$ is square summable. We will prove the following explicit estimate for the 
$\lambda_k = \lambda_k(V)$ based on the Fourier-type coefficients $v_j$.

\begin{theorem}
\label{genpot}
Under the above conditions on the potential $V$, for every $n\in\N$, the $n$th eigenvalue of \eqref{general} with $(a,b) = \R$ satisfies
\begin{equation}
\label{gen}
\dsum_{k=0}^n \lambda_k \leq \dsum_{k=0}^n \frac{2^k (2k)!}{k!}\binom{n+1}{k+1}v_{2k} + \frac{1}{2}(n+1)^2.
\end{equation}
\end{theorem}

\begin{remark}
\label{rem:genpot}
(i) This theorem will be proved by using the eigenfunctions of the quantum harmonic oscillator as test functions in the Rayleigh 
quotient, as was done in Theorem~\ref{mainth1}. The difference is that there we used an estimate for the sum of Hermite polynomials 
resulting from the test functions (Lemma~\ref{lemma:hn}), whereas here we expand out the potential as a Fourier series in Hermite 
polynomials and multiply this against our test functions, in the spirit of the arguments used in \cite{frke}. 
Since the only inequality we use here is that which results from inserting the test functions into the Rayleigh quotient,
and there is no other estimate involved, it follows that the right-hand side of \eqref{gen} must necessarily be smaller than 
the right-hand side of \eqref{comb} if $V$ is of the form $V(x)=x^2+q(x)$ for some $0 \leq q \in L^1(\R)$ (indeed,
it must be equal to the right-hand side of \eqref{basicsum}, i.e.~the middle expression in \eqref{comb}). However,
in practice the two estimates are fundamentally different in nature; for example, it is not easy to see any relation 
between the right-hand side of \eqref{gen} and the trace formula \eqref{photrace}. See also Corollary~\ref{gen2} below.

(ii) As a trivial example to show that the above theorem is sharp, if $V(x)=x^2$, then the only two nonzero coefficients in the 
Fourier expansion of $V$ are $v_2 = 1/4$, $v_0 = 1/2$, and it can easily be seen that~\eqref{gen} reduces to an equality.
\end{remark}

\begin{proof}
As mentioned, we will use the functions $\psi_k(x):= e^{-x^2/2} H_k(x)$ as test functions in the Rayleigh quotient. In order to 
do so, we shall need some more fairly standard facts about integrals of Hermite polynomials $H_k$, which may be found in~\cite{szeg}, for instance: for $n,m \in \N$,
\begin{equation}
\label{double}
\ds \int_{\R}e^{-x^2}H_n(x)H_m(x) \dx{x} = \delta_{mn}\sqrt{\pi}2^n n!
\end{equation}
where $\delta_{jk}$ is the Kronecker delta; and, for $\alpha,\beta,\gamma,s\in\N$ with $\alpha+\beta+\gamma=2s$ even and 
$s\geq \alpha,\beta,\gamma$, we have
\begin{equation}
\label{triple}
\ds \int_{\R}e^{-x^2}H_\alpha(x)H_\beta(x)H_\gamma(x) \dx{x} = \sqrt{\pi}\frac{2^s \alpha!\beta!\gamma!}{(s-\alpha)!
	(s-\beta)!(s-\gamma)!};
\end{equation}
under any other conditions on $\alpha,\beta,\gamma$ and $s$, this integral is $0$. We also note that, combining a standard integration by parts, \eqref{double} and the formula $H_n'(x) = 2n H_{n-1}(x)$, we obtain easily that
\begin{equation}
\label{quadint}
\begin{split}
\ds \int_{\R}e^{-x^2}x^2 H_k^2(x) \dx{x} &= \frac{1}{2}\ds \int_{\R}e^{-x^2}H_k^2(x)\dx{x}
	+2k^2\ds \int_{\R}e^{-x^2}H_{k-1}^2(x)\dx{x}\\
	&=\sqrt{\pi}2^{k-1}k!+\sqrt{\pi}2^k k\, k!
\end{split}
\end{equation}
So, using the $\psi_k$ as test functions, as well the convergence of the $v_j$ to interchange integration and summation (noting that
the functions $V(x), H_k^2(x) \in L^2(\R,e^{-x^2}\dx{x})$, the latter being in $\textrm{span}\{H_0(x),H_2(x),\ldots,H_{2k}(x)\}$) together with \eqref{double},
\begin{displaymath}
\begin{split}
&\dsum_{k=0}^{n} \lambda_{k}  \leq \dsum_{k=0}^{n} \frac{\ds \dint_{\R}\left[\frac{\ds d}{\ds dx}
\left[e^{-x^2/2}H_{k}(x)\right]\right]^2+
e^{-x^2}V(x) H_{k}^{2}(x)\dx{x}}{\ds \dint_{\R}e^{-x^2}H_{k}^{2}(x)\dx{x}} \\
& = \dsum_{k=0}^{n} \left(\lambda_{k}^{0} - \frac{\dint_{\R}e^{-x^2} x^2 H_{k}^{2}(x)\dx{x}}{\int_{\R}e^{-x^2}
	H_k^2(x)\dx{x}}\right)
+ \dsum_{k=0}^{n}\dsum_{j=0}^{\infty}\frac{v_j \dint_{\R} e^{-x^2} (x) H_{j}(x) H_{k}^{2}(x)\dx{x}}{2^k k!\sqrt{\pi}}.
\end{split}
\end{displaymath}
Using \eqref{double} and \eqref{quadint},
\begin{displaymath}
\frac{\dint_{\R}e^{-x^2} x^2 H_{k}^{2}(x)\dx{x}}{\int_{\R}e^{-x^2}
	H_k^2(x)\dx{x}} = k+\frac{1}{2},
\end{displaymath}
while \eqref{triple} with $\alpha=\beta=k$ and $\gamma=j$ implies $\int_{\R} e^{-x^2} (x) H_{j}(x) H_{k}^{2}(x)\dx{x} \neq 0$ if and only if $j$ is even and $j \leq 2k$, and under these conditions, writing $j =: 2m$ for $m=0,\ldots,k$,
\begin{displaymath}
\dint_{\R} e^{-x^2} (x) H_{2m}(x) H_{k}^{2}(x)\dx{x} = \sqrt{\pi}\frac{2^{k+m}(k!)^2(2m)!}{(m!)^2(k-m)!}
= \sqrt{\pi}\frac{2^{k+m}k!(2m)!}{m!}\binom{k}{m}.
\end{displaymath}
Combining the above yields
\begin{displaymath}
\begin{array}{lll}
\dsum_{k=0}^{n} \lambda_{k} & \leq & \dsum_{k=0}^{n} \left(\lambda_k^0-k-\frac{1}{2}\right)+\dsum_{k=0}^n
	\dsum_{m=0}^k \frac{2^m (2m)!}{m!}\binom{k}{m} v_{2m}.
\end{array}
\end{displaymath}
To simplify this last sum, since $\binom{a}{b}=0$ for $b>a$, we may just as well sum $m$ from $0$ to $n$, giving the sum as
\begin{displaymath}
\dsum_{m=0}^n \frac{2^m(2m)!}{m!}v_{2m}\left(\sum_{k=0}^n \binom{k}{m}\right) = \dsum_{m=0}^n
\frac{2^m (2m)!}{m!}\binom{n+1}{m+1} v_{2m},
\end{displaymath}
using a standard formula for binomial coefficients. This establishes the theorem.
\end{proof}

We shall now assume explicitly that the potential $V$ is a perturbation of the harmonic potential and thus return to writing it as $V(x)=x^2+ q(x)$, where we will assume that $q$ is integrable. By adding the terms which are missing in the right-hand side of~\eqref{gen} in order to obtain a sequence which converges to the right-hand side of the trace formula~\eqref{photrace}, and expressing the coefficients in the left-hand side in terms of the Fourier coefficients of the function $q$, we obtain the following result.

\begin{corollary}
 \label{gen2}
 \begin{equation}
\label{gen2eq}
 \begin{array}{ll}
\dsum_{k=0}^n \left[ \lambda_{k}-\lambda_{k}^{0}-\frac{\ds 1}
{\ds \pi \sqrt{\lambda_{k}^{0}}}\dint_{\R} q(x) \dx{x}\right] & \eqskip
 & \hspace*{-3cm}\leq 
\dsum_{k=0}^n \left[\frac{2^k (2k)!}{k!}\binom{n+1}{k+1}q_{2k}-\frac{\ds 1}
{\ds \pi \sqrt{\lambda_{k}^{0}}}\dint_{\R} q(x) \dx{x}\right].
\end{array}
\end{equation}
\end{corollary}

\begin{proof}From $V(x) = q(x) + x^2$ we obtain the relations
\[
 \begin{array}{lll}
  v_{j} & = & \left\{
    \begin{array}{ll}
    q_{0} +\frac{1}{2}, & j=0\eqskip
    q_{2} +\frac{1}{4}, & j=2\eqskip
    q_{k}, & j\neq 0,2.
    \end{array}\right.
 \end{array}.
\]
Replacing this in~\eqref{gen} and adding and subtracting the term
\[
 -\lambda_{k}^{0}-\frac{\ds 1} {\ds \pi \sqrt{\lambda_{k}^{0}}}\dint_{\R} q(x) \dx{x}
\]
inside the summation on the left-hand side of~\eqref{gen}, we obtain, after some manipulations, the desired result.
\end{proof}

\begin{remark}
 Clearly the integral term appearing inside both sums can be cancelled. However, in this way not only do we obtain
 an expression where the left-hand side converges in the limit as $n$ goes to infinity (under additional assumptions 
 on $q$ as in \cite{abp,puso}), but since as noted in Remark~\ref{genpot}(i) the right-hand side of \eqref{gen2eq} 
 is necessarily smaller than the right-hand side of \eqref{maineq} (or \eqref{negeq}, depending on $q$), it follows 
 that it must converge to the right-hand side of the trace formula~\eqref{photrace} and at least as fast as 
 $\bo(1/\sqrt{n})$.
\end{remark}

\section{Power generalizations of Theorems~\ref{mainth1} and \ref{mainth2}}
\label{sec:power}

In this section we generalize the summation bounds obtained in Theorems~\ref{mainth1} and \ref{mainth2} to allow for the 
summands (arranged in various ways) to be raised to a given negative power. We keep the notation and assumptions of 
Sections~\ref{sec:pho} and \ref{sec:neg}, and start with the case where the perturbation $q$ is non-negative.

\begin{theorem}
\label{power1}
Under the assumptions and notation of Theorem~\ref{mainth1}, with $\omega_n$ as in \eqref{omega}, for all $n \geq 0$ and $s>0$, 
\begin{equation}
\label{11}
\left(\frac{1}{n+1}\right)\dsum_{k=0}^n \left(\lambda_k - \lambda_k^0 \right)^{-s}
\geq \left[\frac{\omega_n}{(n+1)\pi}\dint_\R q(x)\dx{x}\right]^{-s}.
\end{equation}
\end{theorem}

Under certain additional assumptions on the potential, we can rearrange the order of the terms in the above bounds somewhat.

\begin{proposition}
\label{power1a}
If $\int_\R q(x)\dx{x} < 32\sqrt{\pi}$, then for all $n\geq 0$ and $s>0$,
\begin{equation}
\label{1a1}
\dsum_{k=0}^n {\lambda_k}^{-s} \geq \dsum_{k=0}^n \left[\lambda_k^0+\frac{\omega_k-\omega_{k-1}}{\pi}
\dint_\R q(x)\dx{x}\right]^{-s}.
\end{equation}
\end{proposition}

We next consider the situation covered by Theorem~\ref{mainth2}, where the perturbation $q$ may take on negative values, provided its negative part decays rapidly enough at infinity. For simplicity, we consider the special case where $q$ has zero mean.

\begin{theorem}
\label{powerzeromean}
Suppose in addition to the assumptions of Theorem~\ref{mainth2} that $\int_{\R} q(x)\dx{x} = 0$. Then for all $n \geq 0$ and $s>0$,
\begin{equation}
\label{zero1}
\dsum_{k=0}^{n} \lambda_k^{-s} \geq (s+1) \dsum_{k=0}^{n} (\lambda_k^0)^{-s} - sq_m\dsum_{k=0}^{n} 
	(\lambda_k^0)^{-s-1} (\varepsilon_k - \varepsilon_{k-1}),
\end{equation}
where $q_m \geq 0$ is defined in Theorem~\ref{mainth2}. Here $\varepsilon_n \geq 0$ is given by \eqref{negasymp} for $n \geq 0$ and we set $\varepsilon_{-1}:=0$.
\end{theorem}

These results will be proved by combining generic results on arbitrary increasing or decreasing sequences of real numbers
(see Lemma~\ref{negs} and what follows it) with the following particular properties of the $\omega_n$.

\begin{lemma}
\label{phoseqs}
The sequence $\{\omega_n\}_{n\in\N}$ is positive and strictly increasing, while $\{\tau_n\}_{n\in\N}$ given by $\tau_n:=\omega_{n+1}-\omega_n$ is positive and non-increasing.
\end{lemma}

\begin{proof}
The $\omega_n$ are obviously all positive. Using the formulae
\begin{displaymath}
\Gamma\left(\frac{z+1}{2}\right) = \frac{z! \sqrt{\pi}}{2^z \left(\frac{z}{2}\right)!}, \qquad \Gamma\left(\frac{z}{2}+1\right)
=\left(\frac{z}{2}\right)!
\end{displaymath}
for $z \in \N$ even, if we assume $n\geq 0$ is even and set
\begin{displaymath}
C_n:=(n+1)\frac{\Gamma\left(\frac{n+1}{2}\right)}{\Gamma\left(\frac{n}{2}+1\right)}= \frac{(n+1)!
\sqrt{\pi}}{2^n \left[\left(\frac{n}{2}\right)!\right]^2} > 0,
\end{displaymath}
then an elementary calculation shows that
\begin{displaymath}
\begin{aligned}
\omega_{n+1}-\omega_n &= \frac{C_n}{2(n+2)}\\
\omega_{n+2}-\omega_{n+1} &= \frac{C_n}{2(n+2)}\\
\omega_{n+3}-\omega_{n+2} &= \frac{(n+3)C_n}{2(n+2)(n+4)},
\end{aligned}
\end{displaymath}
from which we see that $\omega_n$ is increasing in $n$, while $\tau_n = \omega_{n+1}-\omega_n$ is positive and weakly decreasing.
\end{proof}

The following lemma appeared in \cite{frke}, but for the sake of completeness we state and prove it here as well. Here and throughout, we will use the notation $[y]_+$, $y\in\R$, to denote the expression taking on the value $y$ if $y\geq 0$ and zero otherwise; $[f(x)]_{g(x)\geq y}$ will represent $f(x)$ if $g(x) \geq y$ and zero otherwise.

\begin{lemma}
\label{negs}
Suppose the sequences $(a_k)_{k\in\N}$ and $(b_k)_{k\in\N}$ are positive, with $(b_k)_{k\in\N}$ non-decreasing in $k\geq 0$. Suppose also that the sequence $(c_k)_{k\in\N}$ satisfies
\begin{equation}
\label{auxbound}
\dsum_{k=0}^m a_k \leq \dsum_{k=0}^m c_k
\end{equation}
for all $m\geq 0$. Then for all $s>0$ and all $n\geq 0$ we have
\begin{equation}
\label{zeta}
\dsum_{k=0}^n (a_k)^{-s} \geq \dsum_{k=0}^n \left[ (s+1)(b_k)^{-s}-s(b_k)^{-s-1}c_k\right].
\end{equation}
If the sequence $(c_k)_{k\in\N}$ is itself positive and non-decreasing in $k\geq 0$, then the right-hand side of \eqref{zeta} is maximized when $b_k=c_k$ for all $0\leq k\leq n$, in which case \eqref{zeta} simplifies to
\begin{displaymath}
\dsum_{k=0}^n (a_k)^{-s} \geq \dsum_{k=0}^n (c_k)^{-s}.
\end{displaymath}
\end{lemma}

An examination of the proof shows that if we want \eqref{zeta} to hold for some fixed $n\geq 0$, then for the proof to work we 
need \eqref{auxbound} to hold for all $0 \leq m \leq n$.

\begin{proof}
For $\lambda>0$, we use the identity, valid for all $s>0$,
\begin{equation}
\label{powerrep}
	\lambda^{-s} = s(s+1)\dint_0^\infty \alpha^{-s-2}[\alpha-\lambda]_+\dx{\alpha}.
\end{equation}
Hence for $n\geq 0$, $s>0$ arbitrary,
\begin{displaymath}
\begin{split}
\sum_{k=0}^n \left(a_k^{-s}-b_k^{-s}\right) &= s(s+1)\dint_0^\infty \alpha^{-s-2}
	\dsum_{k=0}^n \left([\alpha-a_k]_+-[\alpha-b_k]_+\right)\dx{\alpha}\\
	&\geq s(s+1)\dint_0^\infty \alpha^{-s-2}\dsum_{k=0}^n [b_k-a_k]_{\alpha\geq b_k}\dx{\alpha}\\
	&\geq s(s+1)\dint_0^\infty \alpha^{-s-2}\dsum_{k=0}^n [b_k-c_k]_{\alpha\geq b_k}\dx{\alpha}\\
	&=\dsum_{k=0}^n s(s+1)(b_k-c_k)\dint_{a_k}^\infty \alpha^{-s-2}\dx{\alpha},
\end{split}
\end{displaymath}
which after simplification and rearrangement gives us \eqref{zeta}. For the maximizing property we consider each term on the right-hand side of \eqref{zeta} as a function of $b_k$
\begin{displaymath}
	g_k(b_k):= (s+1)(b_k)^{-s}-s(b_k)^{-s-1}c_k.
\end{displaymath}
Differentiating in $b_k$ shows that $g_k$ reaches its unique maximum when $b_k=c_k$.
\end{proof}

\begin{proof}[Proof of Proposition~\ref{power1a}]
Lemma~\ref{negs} may be applied directly to prove Proposition~\ref{power1a} in the obvious way; for \eqref{1a1}, it merely remains to be confirmed that the sequence
\begin{displaymath}
\left\{\lambda_k^0 + \frac{\omega_k-\omega_{k-1}}{\pi}\dint_\R q(x)\dx{x}\right\}_{k\in\N}
\end{displaymath}
is positive and non-decreasing. Now since $\lambda_{k+1}^0-\lambda_k^0=2$ for all $k\geq 0$, we need $\int_\R q(x)\dx{x}$ (which we assume to be nonzero and hence strictly positive) to be small enough that
\begin{displaymath}
	\omega_{k+2}-2\omega_{k+1}+\omega_k \geq -\frac{2\pi}{\int_\R q(x)\dx{x}}
\end{displaymath}
for all $k\geq 0$. If $k$ is even, then the left-hand side is identically zero, as follows from the proof of Lemma~\ref{phoseqs}.
Otherwise, for $k+1$ odd, we have
\begin{displaymath}
	\omega_{k+3}-2\omega_{k+2}+\omega_{k+1} = \frac{(k+3)C_k}{2(k+2)(k+4)} - \frac{C_k}{2(k+2)},
\end{displaymath}
which, using the definition of $C_k$, may be rearranged to give
\begin{displaymath}
	-\frac{\sqrt{\pi}}{2(k+4)}\cdot\frac{k+1}{k+2}\cdot\frac{k-1}{k-2}\cdots\frac{3}{4}\cdot\frac{1}{2},
\end{displaymath}
which we see is negative and increasing in $k+1\geq 1$ odd. Thus $\omega_{k+3}-2\omega_{k+2}-\omega_{k+1}$ reaches its largest negative value, namely $-C_0/16 = -\sqrt{\pi}/16$, when $k=0$. The requirement on $q(x)$ is therefore that
\begin{displaymath}
	-\frac{\sqrt{\pi}}{16} \geq -\frac{2\pi}{\int_\R q(x)\dx{x}},
\end{displaymath}
that is, we have shown the required sequence is increasing when $\int_\R q(x)\dx{x} \leq 32\sqrt{\pi}$.
\end{proof}

\begin{proof}[Proof of Theorem~\ref{power1}]
To prove \eqref{11} we use a similar idea to the one in Lemma~\ref{negs}, but since the right-hand side of \eqref{maineq}
is not a sequence, the method needs to be adapted slightly to this situation. Namely, starting with the representation \eqref{powerrep} of $\lambda =: \lambda_k - \lambda_k^0$,
\begin{equation}
\label{11proof}
\sum_{k=0}^n (\lambda_k-\lambda_k^0)^{-s} \geq s(s+1)\int_0^\infty \alpha^{-s-2} \sum_{k=0}^n 
	[\alpha - \lambda_k + \lambda_k^0]_{\alpha\geq M} \dx{\alpha}
\end{equation}
for all $M \in \R$; we make the choice $M:= \frac{\omega_n}{(n+1)\pi}\int_\R q(x)\dx{x}$. Using \eqref{maineq}, which, when rearranged, says that
\begin{equation}
\label{rearr}
\sum_{k=0}^n (\lambda_k-\lambda_k^0) \leq \frac{\omega_n}{\pi}\int_\R q(x)\dx{x},
\end{equation}
we have
\begin{displaymath}
\sum_{k=0}^n [\alpha - \lambda_k + \lambda_k^0]_{\alpha\geq \frac{\omega_n}{(n+1)\pi}\int_\R q(x)\dx{x}}
	\geq (n+1)[\alpha-\frac{\omega_n}{(n+1)\pi}\int_\R q(x)\dx{x}]_+.
\end{displaymath}
Substituting this into \eqref{11proof} and applying \eqref{powerrep} yields \eqref{11}.
\end{proof}

\begin{proof}[Proof of Theorem~\ref{powerzeromean}]
This follows directly from Theorem~\ref{mainth2} and Lemma~\ref{negs}, where we take $a_k = \lambda_k$, $b_k = \lambda_k^0$ and $c_k = \lambda_k^0 + (\varepsilon_k - \varepsilon_{k-1})q_m$ (with $\varepsilon_{-1}:=0$).
\end{proof}

\end{document}